\documentclass[11pt]{amsart}
\usepackage{setspace, amsmath, amsthm, amssymb, amsfonts, amscd, epic, graphicx, ulem, dsfont}
\usepackage[T1]{fontenc}
\usepackage{multirow}
\usepackage{bbm}
\usepackage{enumerate}

\makeatletter \@namedef{subjclassname@2010}{
  \textup{2010} Mathematics Subject Classification}
\makeatother

\newtheorem{thm}{Theorem}[section]
\newtheorem{cor}[thm]{Corollary}
\newtheorem{lem}[thm]{Lemma}
\newtheorem{pro}[thm]{Proposition}

\theoremstyle{remark}
\newtheorem*{rema}{Remark}

\theoremstyle{definition}
\newtheorem*{defn}{Definition}
\newtheorem{exa}[thm]{\textbf{Example}}

\newcommand{\REAL}{\text{\rm{Re}}}
\newcommand{\Ima}{\text{\rm{Im}}}

\newcommand{\R}{\mathbb{R}}

\newcommand{\N}{\mathbb{N}}

\newcommand{\C}{\mathbb{C}}

\begin{document}

\title[Absolute Value of Unbounded Operators]{On The Absolute Value of Unbounded Operators}
\author[I. Boucif, S. Dehimi and M. H. Mortad]{Imene Boucif, Souheyb Dehimi and Mohammed Hichem Mortad*}

\dedicatory{}
\thanks{* Corresponding author.}
\date{}


\address{(The first and third authors): Department of
Mathematics, University of Oran 1, Ahmed Ben Bella, B.P. 1524, El
Menouar, Oran 31000, Algeria.\newline {\bf Mailing address}:
\newline Pr Mohammed Hichem Mortad \newline BP 7085 Seddikia Oran
\newline 31013 \newline Algeria}

\email{matheuseama1@gmail.com}

\email{mhmortad@gmail.com, mortad.hichem@univ-oran1.dz}

\address{(The second author): University of Mohamed El Bachir El Ibrahimi, Bordj Bou Arreridj.
Algeria.}

\email{sohayb20091@gmail.com}

\begin{abstract}
The primary purpose of the present paper is to investigate when
relations of the types $|AB|=|A||B|$, $|A\pm B|\leq |A|+|B|$,
$||A|-|B||\leq |A\pm B|$ and $|\overline{\REAL A}|\leq |A|$ (among
others) hold in an unbounded operator setting. As interesting
consequences, we obtain a characterization of (unbounded)
self-adjointness as well as a characterization of invertibility for
the class of unbounded normal operators.
\end{abstract}

\maketitle

\section{Introduction}

All operators considered here are linear but not necessarily
bounded. If an operator is bounded and everywhere defined, then we
say that it belongs to $B(H)$ (the algebra of all bounded linear
operators on $H$).

Most unbounded operators that we encounter are defined on a subspace
(called domain) of a Hilbert space. If the domain is dense, then we
say that the operator is densely defined. In such case, the adjoint
exists and is unique.

In order that the paper be as self-contained as possible, let us
recall a few basic definitions about non-necessarily bounded
operators. If $S$ and $T$ are two linear operators with domains
$D(S)$ and $D(T)$ respectively, then $T$ is said to be an extension
of $S$, written as $S\subset T$, if $D(S)\subset D(T)$ and $S$ and
$T$ coincide on $D(S)$.

An operator $T$ is said to be closed if its graph is closed in
$H\oplus H$. It is called closable if it has a closed extension and
the smallest closed extension of it, is called its closure and it is
denoted by $\overline{T}$ (a standard result states that $T$ is
closable iff $T^*$ has a dense domain and in which case
$\overline{T}=T^{**}$).

If $T$ is closable, then $S\subset T\Rightarrow \overline{S}\subset
\overline{T}$. If $T$ is densely defined, we say that $T$ is
self-adjoint when $T=T^*$; symmetric if $T\subset T^*$; normal if
$T$ is closed and $TT^*=T^*T$.

The product $ST$ and the sum $S+T$ of two operators $S$ and $T$ are
defined in the usual fashion on the natural domains:

\[D(ST)=\{x\in D(T):~Tx\in D(S)\}\]
and
\[D(S+T)=D(S)\cap D(T).\]

In the event that $S$, $T$ and $ST$ are densely defined, then
\[T^*S^*\subset (ST)^*,\]
with the equality occurring when $S\in B(H)$. Also, when $S$, $T$
and $S+T$ are densely defined, then
\[S^*+T^*\subset (S+T)^*,\]
and the equality holding if $S\in B(H)$.

The product of two closed operators need not be closed and the sum
of two closed operators is not closed either. However, it is known
(among other results) that $TS$ and $T+S$ are closed if $T$ is
closed and $S\in B(H)$.

The real and imaginary parts of a densely defined operator $T$ are
defined respectively by
\[\REAL T=\frac{T+T^*}{2} \text{ and } \Ima T=\frac{T-T^*}{2i}.\]
Clearly, if $T$ is closed, then $\REAL T$ is symmetric but it is not
always self-adjoint. Moreover, it may also fail to be closed.

If $S\in B(H)$ and $T$ is unbounded, then $S$ commutes with $T$ if
$ST\subset TS$. If $S$ and $T$ are normal, then the previous amounts
to the commutativity of their spectral measures (see e.g.
\cite{Arai-2005-Rev-math-phys}) or \cite{SCHMUDG-book-2012}). Two
unbounded normal operators are said to (strongly) commute if their
corresponding spectral measures commute.

If $T$ is such that $<Tx,x>\geq0$ for all $x\in D(T)$, then we say
that $T$ is positive. We can define the unique positive self-adjoint
square root which we denote by $\sqrt T$ (see
\cite{Sebestyen-Tarcsay-square root-NEW} for a new proof of the
existence of the positive square root of unbounded positive
self-adjoint operators). If $T$ is positive and $ST\subset TS$ where
$S\in B(H)$, then $S\sqrt T\subset \sqrt T S$ (see \cite{Bernau
JAusMS-1968-square root}). Also, if $ST\subset TS$ where $S\in B(H)$
is positive, then $\sqrt ST\subset T\sqrt S$ (see e.g.
\cite{Mortad-PAMS2003}).

If $T$ is densely defined and closed, then $T^*T$ (and $TT^*$) is
self-adjoint and positive (a celebrated result due to von-Neumann,
see e.g. \cite{SCHMUDG-book-2012}). We digress a little bit to say
that the self-adjointness of $T^*T$ (say) alone does not imply that
$A$ is closed. A simple counterexample is prescribed in
\cite{Sebestyen-Tarcsay-TT* von Neumann T closed}.
Sebestyén-Tarcsay went on in the same reference to show that the
self-adjointness of \textit{both }$T^*T$ and $TT^*$ implies the
closedness of $T$. A re-proof may be found in
\cite{Gesztesy-Schm\"{u}dgen-2018}.

As just observed, if $T$ is closed then $T^*T$ is self-adjoint and
positive. Hence it is legitimate to define its square root. The
unique positive self-adjoint square root of $T^*T$ is denoted by
$|T|$. It is customary to call it the absolute value or modulus of
$T$. The absolute value plays a prominent role in Operator Theory.
We just give two uses: It intervenes in any version of the polar
decompositions of an operator. We also utilize it when defining
singular values in Matrix Theory.

As alluded, we are mainly (but not only) interested in investigating
when relations of the types
\[|AB|=|A||B|, ~|A\pm B|\leq |A|+|B|, ~||A|-|B||\leq |A\pm B|\text{ and }|\overline{\REAL A}|\leq |A|\]
in an unbounded setting. The all-bounded case has already been
treated in \cite{Mortad-abs-value-BOUNDED}. So, there is no need to
reiterate that the previous inequalities, without any a priori
conditions, are false in general (see e.g.
\cite{Mortad-Oper-TH-BOOK-WSPC}).

If $T$ is closed, then (see e.g. Lemma 7.1 in
\cite{SCHMUDG-book-2012})
\[D(T)=D(|T|)\text{ and } \|Tx\|=\||T|x|\|,~\forall x\in D(T).\]
It is also plain that if $T$ is normal, then $|T|=|T^*|$. Also, in
the event of the normality of $T$, $\overline{\REAL T}$ and
$\overline{\Ima T}$ are self-adjoint (see e.g. \cite{WEI}).

The use of the Fuglede (-Putnam) Theorem is unavoidable when working
with normal operators. For convenience, we recall it next (see e.g.
\cite{Con} for a proof):

\begin{thm}\label{Fuglede Putnam classic}
If $T$ is a bounded operator and $N$ and $M$ are non necessarily
bounded normal operators, then
\[TN\subset MT \Longrightarrow TN^*\subset M^*T.\]
\end{thm}

There is no harm in recalling some more definitions. The following
definition will be used often.

\begin{defn}\label{defn pos unbd schmud}
Let $T$ and $S$ be unbounded positive self-adjoint operators. We say
that $S\geq T$ if $D(S^{\frac{1}{2}})\subseteq D(T^{\frac{1}{2}})$ and $%
\left\Vert S^{\frac{1}{2}}x\right\Vert \geq \left\Vert T^{\frac{1}{2}%
}x\right\Vert$ for all $x\in D(S^{\frac{1}{2}}).$
\end{defn}

The Heinz Inequality is valid for positive unbounded self-adjoint
operators as well. That is and in particular, if $S$ and $T$ are
self-adjoint operators, then (see e.g. \cite{SCHMUDG-book-2012})
\[S\geq T\geq0\Longrightarrow \sqrt S\geq \sqrt T.\]
Also, as in Page 200 in \cite{WEI}, if $S$ and $T$ are self-adjoint,
$T$ is boundedly invertible and $S\geq T\geq0$, then $S$ is
boundedly invertible and $S^{-1}\leq T^{-1}$.

Finally, we recall the definition of an unbounded hyponormal
operator.

\begin{defn}
A densely defined operator $T$ with domain $D(T)$ is called
hyponormal if
\[D(T)\subset D(T^*)\text{ and } \|T^*x\|\leq\|Tx\|,~\forall x\in D(T).\]
\end{defn}

As in the bounded case, we have (the proof may be found in
\cite{Dehimi-Mortad-REID}):

\begin{pro}
Let $T$ be a closed hyponormal operator. Then
\[T^*T\geq TT^*.\]
\end{pro}

The following recently proved result will be needed in some of our
proofs.

\begin{thm}\label{Dehimi-Mortad-HYPONORMAL-ABS-VALUE-THM}(\cite{Dehimi-Mortad-REID})
Let $T$ be a closed hyponormal operator. Then
\[
|<Tx,x>|\leq <|T|x,x>\text{ for all }x\in D(T).
\]
\end{thm}

Despite the fact that we have recalled most results which will be
used in our proofs, we refer readers to \cite{SCHMUDG-book-2012} and
\cite{Simon-Operator-Theory-NEW} for further reading. We also refer
readers to \cite{Birman-Solomjak-sepctral theorem of s.a. operators}
or \cite{SCHMUDG-book-2012} for the Spectral Theorem.

\section{The Absolute Value and Products}

We start with the following known result which is recalled for
readers convenience.

\begin{pro}\label{sqrt AB=sqrt A sqrt B PRO}
Let $A$ be a self-adjoint positive operator and let $B\in B(H)$ be
positive (hence also self-adjoint). If $BA\subset AB$, then $AB$
(and $\overline{BA}$) is self-adjoint, positive and
$\overline{BA}=AB$. Moreover,
\[\sqrt{AB}=\sqrt A\sqrt B.\]
\end{pro}

\begin{proof}The self-adjointness and the positiveness of $AB$
follow by invoking the Spectral Theorem as carried out in Lemma 3.1
of \cite{Jung-Mortad-Stochel}. The equality $\sqrt{AB}=\sqrt A\sqrt
B$ may also be obtained via a similar argument. Alternatively, we
may proceed as follows: We have
\[BA\subset AB\Longrightarrow \sqrt BA\subset A\sqrt B \Longrightarrow \sqrt B\sqrt A\subset \sqrt A\sqrt B.\]
Now, by the first part of the proof, $\sqrt A\sqrt B$ is
self-adjoint (and positive) and hence so is $(\sqrt A\sqrt B)^2$.
But,
\[(\sqrt A\sqrt B)^2=\sqrt A\sqrt B\sqrt A\sqrt B\subset \sqrt A\sqrt A\sqrt B\sqrt B=AB.\]
Next, as both $(\sqrt A\sqrt B)^2$ and $AB$ are self-adjoint, then
by the maximality of self-adjoint operators (see e.g.
\cite{SCHMUDG-book-2012}, cf. \cite{Meziane-Mortad}), we obtain
\[(\sqrt A\sqrt B)^2=AB.\]

Accordingly, by the uniqueness of the positive square root, we infer
that
\[\sqrt{AB}=\sqrt A\sqrt B,\]
as needed.
\end{proof}

The following perhaps known result will be used without further
notice. Its power lies in the fact that $B$ is not assumed to be
normal (and so it cannot be established using the Spectral Theorem).

\begin{pro}\label{abs A abs B abs B abs A PRO}
Let $A$ be normal and let $B\in B(H)$. Then
\[BA\subset AB\Longrightarrow |B| |A|\subset |A| |B|,\]
i.e. $|B| |A|x=|A| |B|x$ whichever $x\in D(B|A|)=D(|A|)=D(A)$.
\end{pro}

\begin{proof}
Since $BA\subset AB$ and $A$ is normal, we get $BA^*\subset A^*B$.
Hence
\[BA^*A\subset A^*BA\subset A^*AB.\]
Therefore, $B|A|\subset |A|B$. As $|A|$ is self-adjoint, it follows
by taking adjoints that $B^*|A|\subset |A|B^*$. Thus,
\[B|A|\subset |A|B\Longrightarrow B^*B|A|\subset B^*|A|B\subset |A|B^*B.\]
Accordingly,
\[|B| |A|\subset |A| |B|,\]
as required.
\end{proof}

\begin{cor}\label{corollary A*AB*B self-adjoint}
Let $A$ be normal and let $B\in B(H)$. If $BA\subset AB$, then
$A^*AB^*B$ is self-adjoint.
\end{cor}

\begin{proof}
As above, we obtain $B^*BA^*A\subset A^*AB^*B$. Since $B^*B$ and
$A^*A$ are self-adjoint, Proposition \ref{sqrt AB=sqrt A sqrt B PRO}
implies the self-adjointness of $A^*AB^*B$.
\end{proof}

\begin{thm}\label{abs value AB MAIN THM}
Let $A$ be normal and let $B\in B(H)$. If $BA\subset AB$, then
\[|\overline{BA}|=|AB|=|A||B|=\overline{|B||A|}.\]
\end{thm}

\begin{proof}The equality $|A||B|=\overline{|B||A|}$ is established as
follows: By Proposition \ref{abs A abs B abs B abs A PRO}, we have
$|B| |A|\subset |A| |B|$. Since $|B|$ and $|A|$ are self-adjoint, it
follows by Proposition \ref{sqrt AB=sqrt A sqrt B PRO} that
\[\overline{|B||A|}=|A||B|.\]

Now, by the general theory, $BA\subset AB$ gives $B^*A^*\subset
A^*B^*$. Also and as above, the normality of $A$ yields
\[BA\subset AB\Longrightarrow BA^*\subset A^*B.\]
Hence by passing to adjoints, $B^*A\subset AB^*$. Therefore,
\[(AB)^*AB\supset B^*A^*AB=B^*AA^*B\supset B^*ABA^*\supset B^*BAA^*.\]
The closedness of $AB$ gives the self-adjointness of $(AB)^*AB$. The
previous "inclusion" becomes after taking adjoints
\[(AB)^*AB\subset AA^*B^*B=A^*AB^*B.\]
Since $(AB)^*AB$ and $A^*AB^*B$ are self-adjoint, by the maximality
of self-adjoint operators, we get
\[(AB)^*AB=A^*AB^*B.\]
We infer by a glance at Proposition \ref{sqrt AB=sqrt A sqrt B PRO}
that
\[|AB|=\sqrt{(AB)^*AB}=\sqrt{A^*AB^*B}=\sqrt{A^*A}\sqrt{B^*B}=|A||B|.\]
To prove the last equality, we proceed as before. We have

\[(BA)^*BA=A^*B^*BA\supset B^*A^*BA\supset B^*BA^*A.\]
Since $BA$ is closeable and $B^*B\in B(H)$, it follows by passing to
adjoints that
\[(BA)^*\overline{BA}\subset [(BA)^*BA]^*\subset A^*AB^*B.\]
As above, we obtain
\[(BA)^*\overline{BA}=A^*AB^*B.\]
Accordingly,
\[|\overline{BA}|=\sqrt{(\overline{BA})^*\overline{BA}}=\sqrt{(BA)^*\overline{BA}}=\sqrt{A^*AB^*B}=|A||B|.\]
\end{proof}

\begin{rema}The condition $BA\subset AB$ may be obtained if $AB$ is normal, $A$ and $B$ are self-adjoint (one of them is
positive) and $B\in B(H)$ (see \cite{Mortad-PAMS2003},
\cite{Gustafson-Mortad-2016}, cf. \cite{Jung-Mortad-Stochel}).
\end{rema}

\begin{cor}\label{abs value corollary AB B*A}
Let $A$ be a normal operator and let $B\in B(H)$. If $BA\subset AB$,
then
\[|AB|=|A^*B|=|\overline{BA}|=|\overline{BA^*}|.\]
\end{cor}

\begin{cor}
Let $A$ be a boundedly invertible normal operator. Then
\[|A^{-1}|=|A|^{-1}.\]
\end{cor}

\begin{proof}
We may use the Spectral Theorem. Otherwise, by hypothesis,
\[A^{-1}A\subset AA^{-1}=I\]
where $A^{-1}\in B(H)$. Hence
\[|A||A^{-1}|=|AA^{-1}|=|I|=I.\]
Since $|A|$ is self-adjoint and right invertible, it is invertible
(see \cite{Dehimi-Mortad-2018}) and clearly
\[|A^{-1}|=|A|^{-1}.\]
\end{proof}

By calling on the Spectral Theorem for (unbounded) normal operators,
it is easy to see that $|A^n|=|A|^n$ for any $n\in\N$ and where $A$
is an unbounded normal operator. As a consequence of a recent result
due to J. Jab{\l}o\'{n}ski-Jung-Stochel in \cite{Jablonski et al
2014}, we may extend the above to quasinormal operators (a class for
which there is no Spectral Theorem). For readers convenience, we
recall that a closed densely defined operator $A$ is said to be
quasinormal if $U|A|\subset |A|U$ where $A=U|A|$ is the
(non-unitary) polar decomposition of $A$.

\begin{pro}(Lemma 3.5, \cite{Jablonski et al 2014})
If $A$ is unbounded and quasinormal, then for all $n\in\N$
\[(A^*A)^n=A^{*n}A^n=(A^n)^*A^n.\]
\end{pro}

Hence, we have:

\begin{cor}
If $A$ is an unbounded quasinormal operator, then
\[|A^n|=|A|^n\]
for all $n\in\N$.
\end{cor}

The previous corollary is not true for the weaker class of
hyponormal operators even bounded ones. Indeed, we have:

\begin{exa}
Let $S\in B(\ell^2)$ be the usual (unilateral) shift. Remember that
$SS^*\neq I$ and $S^*S=I$ where $I$ is the identity operator on
$\ell^2$. Now, take $A=S+I$ so that $A$ is hyponormal. Then
\[|A|^2=|S+I|^2=(S^*+I)(S+I)=2I+S+S^*.\]
On the other hand,
\[|A^2|=\sqrt{(S^*+I)^2(S+I)^2}=\sqrt{6I+4S^*+4S+S^{*2}+S^2}.\]
If $|A^2|=|A|^2$ held, then we would obtain $|A^2|^2=|A|^4$. Working
out details would then yield $SS^*=I$ which is impossible.
\end{exa}

Before passing to inequalities, we generalize Theorem \ref{abs value
AB MAIN THM} to a finite family of normal operators. First, we have:

\begin{pro}\label{pppp}
Let $A$ and $B$ be two strongly commuting normal operators. Then
$\overline{AB}=\overline{BA}$ are normal and
\[|\overline{BA}|=|\overline{AB}|=\overline{|A||B|}=\overline{|B||A|}.\]
\end{pro}

\begin{proof}
The proof makes use of the Spectral Theorem. Since $A$ and $B$ are
normal, by the Spectral Theorem we may write
\[A=\int_{\C}zdE_A \text{ and } B=\int_{\C}z'dF_B,\]
where $E_A$ and $F_B$ designate the associated spectral measures. By
the strong commutativity, we have
\[E_A(I)F_B(J)=F_B(J)E_A(I)\]
for all Borel sets $I$ and $J$ in $\C$. Hence
\[E_{A,B}(z,z')=E_A(z)F_B(z')\]
defines a two parameter spectral measure. Thus
\[C=\int_{\C}\int_{\C}zz'dE_{A,B}\]
defines a normal operator, such that
$C=\overline{AB}=\overline{BA}$. Therefore, as $|zz'|=|z||z'|$ for
all $z,z'$, then (see e.g. Page 78 of \cite{SCHMUDG-book-2012})
\[|\overline{AB}|=\int_{\R}\int_{\R}|zz'|dE_{A,B}=\overline{|A||B|},\]
as wished.
\end{proof}

\begin{cor}
Let $A$ and $B$ be two strongly commuting normal operators. Then
\[|\overline{AB}|=|\overline{A^*B}|=|\overline{AB^*}|=|\overline{A^*B^*}|=|\overline{BA}|=|\overline{BA^*}|=|\overline{B^*A}|=|\overline{B^*A^*}|.\]
\end{cor}

\begin{proof}
The proof follows easily from Proposition \ref{pppp} and from
$|A^*|=|A|$ and $|B^*|=|B|$.
\end{proof}

Using the polar decomposition of (unbounded) normal operators, the
following result then becomes obvious.

\begin{cor}
Let $A$ and $B$ be two strongly commuting normal operators. Then
\[\overline{BA}=\overline{AB}=U\overline{|A||B|}=U\overline{|B||A|}=\overline{|B||A|}U=\overline{|A||B|}U\]
for some unitary operator $U\in B(H)$.
\end{cor}

Before carrying on, we give the following probably known result on
the important notion of the adjoint of products:

\begin{pro}\label{(AC)*=....Normal PRO}
Let $A$ and $B$ be two strongly commuting normal operators. Then
\[(AB)^*=(\overline{AB})^*=\overline{B^*A^*}=\overline{A^*B^*}=(BA)^*.\]
\end{pro}

\begin{proof}
The proof is contained in the proof of Proposition \ref{pppp} by
remembering that
\[A^*=\int_{\C}\overline{z}dE_A \text{ and } B^*=\int_{\C}\overline{z'}dF_B.\]
\end{proof}

\begin{pro}\label{ABC MODULUS PRO}
Let $A$ and $C$ be strongly commuting normal operators and let $B\in
B(H)$. If $BA\subset AB$, $BC\subset CB$ and $AC$ is densely
defined, then
\[|\overline{ABC}|=|\overline{AC}||B|=\overline{|A||C|}|B|.\]
\end{pro}

The proof of the preceding proposition relies upon a Fuglede-like
result.

\begin{lem}\label{fuglede-type closure normal LEM}
Let $B\in B(H)$ and let $T$ be a densely defined operator such that
$\overline{T}$ is normal. If $BT\subset TB$, then
\[BT^*\subset T^*B\text{, } B\overline{T}\subset \overline{T}B \text{ and } B^*\overline{T}\subset \overline{T}B^*.\]
\end{lem}

\begin{proof}
Since $\overline{T}$ is normal, $\overline{T}^*=T^*$ stays normal.
Now,
\begin{align*}
BT\subset TB\Longrightarrow& B^*T^*\subset T^*B^* \text{ (by taking adjoints)}\\
\Longrightarrow &B^*\overline{T}\subset \overline{T}B^* \text{ (use the Fuglede Theorem)}\\
\Longrightarrow& BT^*\subset T^*B \text{ (by taking adjoints)}\\
\Longrightarrow& B\overline{T}\subset \overline{T}B \text{ (apply
the Fuglede Theorem again).}
\end{align*}
This marks the end of the proof.
\end{proof}

\begin{rema}
It is worth noting that (see Lemma 2.8 of \cite{Tian-Jorgensen PhD
THESIS}): \textit{If $B\in B(H)$ and $T$ is densely defined such
that $\overline{T}$ exists, then
\[BT\subset TB\Longleftrightarrow B\overline{T}\subset \overline{T}B.\]}
\end{rema}

Now, we give a proof of Proposition \ref{ABC MODULUS PRO}.

\begin{proof}
By assumptions,
\[B(AC)\subset ABC\subset (AC)B.\]
Since $AC$ is densely defined and $D(AC)=D(BAC)\subset D(ABC)$,
clearly $ABC$ is densely defined. Hence
\[(ABC)^*\subset [B(AC)]^*=(AC)^*B^*.\]
As $\overline{AC}$ is normal, then Lemma \ref{fuglede-type closure
normal LEM} gives
\[B(AC)^*\subset (AC)^*B \text{ or } B^*\overline{AC}\subset \overline{AC}B^*.\]
Also,
\[ABC\subset ACB\subset \overline{AC}B\]
and so
\[\overline{ABC}\subset\overline{AC}B\]
for $\overline{AC}B$ is closed. Hence, we may write
\[(ABC)^*\overline{ABC}\subset (AC)^*B^*\overline{ABC}\subset (AC)^*B^*\overline{AC}B\subset (AC)^*\overline{AC}B^*B.\]
Corollary \ref{corollary A*AB*B self-adjoint} gives the
self-adjointness of $(AC)^*\overline{AC}B^*B$ and because
$\overline{ABC}$ is closed, we then get
\[(ABC)^*\overline{ABC}=(AC)^*\overline{AC}B^*B.\]
Ergo, Proposition \ref{sqrt AB=sqrt A sqrt B PRO} (and Corollary
\ref{corollary A*AB*B self-adjoint} again) yields
\[|\overline{ABC}|=|\overline{AC}||B|=\overline{|A||C|}|B|,\]
establishing the result.
\end{proof}

\begin{cor}
Let $A$ be a normal operator and $B\in B(H)$. If $BA\subset AB$,
then
\[|\overline{ABA}|=|A|^2|B|=|A|\overline{|B||A|}.\]
\end{cor}

\begin{cor}
Let $(A_i)_{i=1,\cdots,n}$ be pairwise strongly commuting normal
operators. Then
\[|\overline{A_1A_2\cdots A_n}|=\overline{|A_1||A_2|\cdots|A_n|}.\]
\end{cor}

\begin{proof}
It follows by induction using $n$ parameter spectral measures.
\end{proof}

We finish with an ultimate generalization:

\begin{pro}
Let $(A_i)_{i=1,\cdots,n}$ be pairwise strongly commuting normal
operators such that $A_1A_2\cdots A_n$ is densely defined. Let $B\in
B(H)$ be such that $BA_i\subset A_iB$ for $i=1,2,\cdots,n$. Then
\[|\overline{BA_1A_2\cdots A_n}|=|\overline{A_1A_2\cdots A_n}B|=|\overline{A_1A_2\cdots A_i BA_{i+1}\cdots A_n}|=\overline{|A_1||A_2|\cdots|A_n|}|B|.\]
\end{pro}

\begin{proof}These equalities are simple consequences of the
results obtained above. For example, to prove the last equality,
observe that $A_1A_2\cdots A_i$ and $A_{i+1}\cdots A_n$ are normal.
Then
\[|\overline{A_1A_2\cdots A_i BA_{i+1}\cdots A_n}|=|\overline{(A_1A_2\cdots A_i)(A_{i+1}\cdots A_n)}||B|=\overline{|A_1||A_2|\cdots|A_n|}|B|,\]
as required.
\end{proof}

\section{The Absolute Value, Sums and Inequalities}

\begin{thm}\label{thm triangle inequality MAIN}
Let $A$ be a normal operator and let $B\in B(H)$ be hyponormal. If
$BA\subset AB$, then
\[|A+B|\leq |A|+|B|.\]
\end{thm}

The proof requires the following result which is also interesting in
its own.

\begin{thm}\label{A^*B hyponormal THM}
Let $A$ be a normal operator and let $B\in B(H)$ be hyponormal. If
$BA\subset AB$, then $A^*B$ is hyponormal, that is,
\[\|(A^*B)^*x\|\leq \|A^*Bx\|\]
for all $x\in D(A^*B)\subset D[(A^*B)^*]$.
\end{thm}

\begin{proof}
Since $BA\subset AB$, by the Fuglede Theorem $BA^*\subset A^*B$
because $A$ is normal. By Corollary \ref{abs value corollary AB B*A}
we have
\[|A^*B|=|\overline{BA^*}|.\]
Since $A^*B$ is closed (by the general theory), we obtain
\[D(A^*B)=D(|A^*B|)=D(|\overline{BA^*}|)=D(\overline{BA^*}).\]
On the other hand, as $A^*B$ is densely defined, we always have
\[B^*A^{**}=B^*A\subset(A^*B)^*.\]
Hence
\[\overline{B^*A}\subset(A^*B)^*\]
and so $D(A^*B)=D(\overline{B^*A})\subset D[(A^*B)^*]$.

Now, for all $x\in D(A)=D(A^*)=D(BA^*)\subset D(A^*B)\subset
D[(A^*B)^*]$, we have (using the hyponormality of $B$)
\[\|(A^*B)^*x\|=\|B^*Ax\|\leq \|BAx\|=\|ABx\|=\|A^*Bx\|,\]
proving the hyponormality of $A^*B$, as wished.
\end{proof}

\begin{cor}
Let $A$ be a normal operator and let $B\in B(H)$ be hyponormal. If
$BA\subset AB$, then $AB$ is hyponormal.
\end{cor}

\begin{proof}
Clearly, $BA\subset AB$ gives $BA^*\subset A^*B$. Now the foregoing
theorem yields the hyponormality of $A^{**}B=AB$.
\end{proof}

We are ready to prove Theorem \ref{thm triangle inequality MAIN}:

\begin{proof}First, as usual we have $BA\subset AB$ and $BA^*\subset
A^*B$. Whence, if $x\in D(A)=D(A^*)$, then $Bx\in D(A)=D(A^*)$.
Moreover, the self-adjointness of $|A|$ and $|B|$ gives the
self-adjointness (and positiveness) of $|A|+|B|$ given that $|B|\in
B(H)$. Since $A$ is closed and $B\in B(H)$, $A+B$ is closed. Hence
$|A+B|$ makes sense and besides
\[D(|A|+|B|)=D(|A|)=D(A)=D(A+B)=D(|A+B|).\]
Now, since $|A+B|$ is self-adjoint and positive, to show the
required triangle inequality, by calling on the Heinz Inequality
(unbounded version), it suffices to show that
\[|A+B|^2\leq (|A|+|B|)^2.\]

Let $x\in D(A)$. Then
\begin{align*}
\||A+B|x\|^2&=\|(A+B)x\|^2 \\
&=<(A+B)x,(A+B)x>\\
&=\|Ax\|^2+\|Bx\|^2+<Bx,Ax>+<Ax,Bx>\\
&=\|Ax\|^2+\|Bx\|^2+<Bx,Ax>+\overline{<Bx,Ax>}\\
&=\|Ax\|^2+\|Bx\|^2+2\REAL<Bx,Ax>\\
&=\|Ax\|^2+\|Bx\|^2+2\REAL<A^*Bx,x>\\
&\leq \|Ax\|^2+\|Bx\|^2+2|<A^*Bx,x>|.
\end{align*}
But $A^*B$ is closed. It is also hyponormal by Theorem \ref{A^*B
hyponormal THM} and so Theorem
\ref{Dehimi-Mortad-HYPONORMAL-ABS-VALUE-THM} yields
\[|<A^*Bx,x>|\leq <|A^*B|x,x>\]
for each $x$. By Proposition \ref{abs A abs B abs B abs A PRO} and
Theorem \ref{abs value AB MAIN THM}, we have for each $x\in D(A)$:
\[|A^*B|x=|A^*||B|x=|A||B|x=|B||A|x.\]
Hence
\[2|<A^*Bx,x>|\leq 2<|A^*B|x,x>=<|A||B|x,x>+<|B||A|x,x>.\]
Accordingly,
\[\||A+B|x\|^2\leq \|Ax\|^2+\|Bx\|^2+<|A||B|x,x>+<|B||A|x,x>=\|(|A|+|B|)x\|^2\]
or merely
\[\||A+B|x\|\leq \|(|A|+|B|)x\|\]
for all $x\in D(A)$. By Definition \ref{defn pos unbd schmud}, this
just means that
\[|A+B|^2\leq (|A|+|B|)^2,\]
as needed above.
\end{proof}

\begin{cor}
Let $A$ be a normal operator and let $B\in B(H)$ be hyponormal. If
$BA\subset AB$, then
\[|A-B|\leq |A|+|B|.\]
\end{cor}

\begin{rema}
The proof of the previous theorem in the bounded case (as it
appeared in \cite{Mortad-abs-value-BOUNDED}) used the known fact
that
\[|\REAL T|\leq |T|...(1)\]
whenever $T\in B(H)$ is hyponormal. So, the proof above is a new
proof of Theorem 3.4 in the above reference. The natural question is
whether we can expect Inequality $(1)$ to hold for unbounded closed
hyponormal operators? The answer is no in general, however, we have
the following related and interesting result:
\end{rema}

\begin{pro}\label{PP}
Let $A$ be a closed hyponormal operator. Then
\[\left|\overline{\REAL A}\right|\leq |A|\]
where $\overline{\REAL A}$ denotes the closure of $\REAL A$.
\end{pro}

\begin{proof}
Since $A$ is hyponormal,  we have $D(A)\subset D(A^*)$ and so
$A+A^*$ is densely defined. Since $A$ is closed, $A+A^*$ is
symmetric. Therefore, $\REAL A\subset \overline{\REAL A}$. Let $x\in
D(A)$. Then
\[(\REAL A)x=(\overline{\REAL A})x.\]
 As above, it suffices to show that
\[\left|\overline{\REAL A}\right|^2\leq |A|^2.\]
By the closedness of $A$,
\[D(|A|)=D(A)=D(\REAL A)\subset D(\overline{\REAL A})=D(|\overline{\REAL A}|).\]
Let $x\in D(|A|)$. Using the hyponormality of $A$ implies that
\[\||\overline{\REAL A}|x\|=\|(\overline{\REAL A})x\|=\|(\REAL A)x\|\leq \frac{1}{2}(\|Ax\|+\|A^*x\|)\leq \|Ax\|=\||A|x\|,\]
i.e. $\left|\overline{\REAL A}\right|^2\leq |A|^2$. Using the Heinz
Inequality, we infer that
\[\left|\overline{\REAL A}\right|\leq |A|,\]
as coveted.
\end{proof}

Mutatis mutandis, we may prove:

\begin{pro}\label{qq}
Let $A$ be a closed hyponormal operator. Then
\[\left|\overline{\Ima A}\right|\leq |A|\]
where $\overline{\Ima A}$ denotes the closure of $\Ima A$.
\end{pro}

The following result gives an important application to
self-adjointness.  Notice that its first proof was a little longer.
Then, we came across the paper \cite{JOCIC} and the proof became
shorter. The next theorem is also a partial answer in an unbounded
setting to a famous conjecture posed by Fong-Tsui in \cite{Fong-Tsui
1981} (asked in the $B(H)$ context). The analogous result on $B(H)$
was already obtained in \cite{Mortad-Fong-TSUI-VERY-FIRST-PAper}.

\begin{thm}\label{56}
Let $A$ be a normal operator such that
\[|A|\leq |\overline{\REAL A}|.\]
Then $A$ is self-adjoint.
\end{thm}

\begin{proof}
Since $A$ is normal, it is closed and hyponormal and so by
Proposition \ref{PP}, $|\overline{\REAL A}|\leq |A|$. A glance at
the hypothesis of the theorem yields
\[|A|=|\overline{\REAL A}|.\]
By the normality of $A$, we have $D(A)=D(A^*)$. Since
$\overline{\REAL A}$ is self-adjoint, we may write
\[D[(\REAL A)^*]=D[(\overline{\REAL A})^*]=D(\overline{\REAL A})=D(|A|)=D(A)\]
and
\[\|A^*x\|=\|Ax\|=\||A|x\|=\||\overline{\REAL A}|x\|=\|\overline{\REAL A}x\|=\|(\REAL A)^*x\|.\]
Therefore, the conditions of Theorem 3 in \cite{JOCIC} are fulfilled
and so $A=A^*$, as needed.
\end{proof}

As alluded to above, the power of Theorem \ref{thm triangle
inequality MAIN} say, lies in the fact that one operator is
hyponormal and so there is no room for the (two parameter) Spectral
Theorem. If we work with strongly commuting normal operators (even
both unbounded), then the proof becomes a slightly simpler. Besides,
this will be used to generalize some of the results above.

\begin{thm}\label{abs A+B both unbd normal TTHM}
Let $A$ and $B$ be two strongly commuting unbounded normal
operators. Then,
\[|\overline{A+B}|\leq |A|+|B|.\]
\end{thm}

\begin{proof}
Since $A$ and $B$ are normal, by the Spectral Theorem we may write
\[A=\int_{\C}zdE_A\text{ and } B=\int_{\C}z'dF_B,\]
where $E_A$ and $F_B$ denote the associated spectral measures. As
$A$ and $B$ strongly commute, then so do $|A|$ and $|B|$. Since
$|A|$ and $|B|$ are also self-adjoint and positive, by Lemma 4.15.1
in \cite{Putnam-book-1967}, it follows that $|A|+|B|$ is
self-adjoint (hence closed) and positive. As for domains, we have
\[D(|A|+|B|)=D(|A|)\cap D(|B|)=D(A)\cap D(B)=D(A+B)\subset D(\overline{A+B})\subset D(|\overline{A+B}|).\]

By the strong commutativity, we have
\[E_A(I)F_B(J)=F_B(J)E_A(I)\]
for all Borel sets $I$ and $J$ in $\C$. Hence
\[E_{A,B}(z,z')=E_A(z)F_B(z')\]
defines a two parameter spectral measure. Thus
\[C=\iint\limits_{\C^2}(z+z')dE_{A,B}\]
defines a normal operator such that $C=\overline{A+B}$. Since
$|z+z'|\leq |z|+|z'|$ for all $z,z'$, it follows that
\[|\overline{A+B}|=\iint\limits_{\R^2}|z+z'|dE_{A,B}(z,z')\leq |A|+|B|\]
(we used the positivity of the measure $<E_{A,B}(\Delta)x,x>$ with
$\Delta$ being a Borel subset of $\R^2$). In fact, what we have
proved so far is "only"
\[|\overline{A+B}|\preceq
|A|+|B|\] where $\preceq$ is defined as in Definition 10.5 (Page
230) in \cite{SCHMUDG-book-2012}. In other language, we have shown
that
\[<|\overline{A+B}|x,x>\leq <(|A|+|B|)x,x>\]
for all $x\in D(|A|+|B|)\subset D(|\overline{A+B}|)$.

 Since $|\overline{A+B}|$ and
$|A|+|B|$ are self-adjoint and positive, by Lemma 10.10 (again in
\cite{SCHMUDG-book-2012}), "$\preceq$" becomes "$\leq $", that is,
we have established the desired inequality:
\[|\overline{A+B}|\leq |A|+|B|.\]
\end{proof}

\begin{rema}
The strong commutativity of the normal $A$ and $B$ is not sufficient
to make the sum $A+B$ closed even when $A$ and $B$ are self-adjoint
as shown by the usual trick: Consider an unbounded self-adjoint
operator $A$ with a \textit{non closed} domain $D(A)$. Then $B=-A$
is also self-adjoint and evidently strongly commutes with $A$ but
the bounded $A+B=0$ on $D(A)$ is not closed.
\end{rema}

\begin{cor}\label{triangle inequality normal real imag cor}
Let $A$ be a normal operator. Then
\[|A|\leq |\overline{\REAL A}|+|\overline{\Ima A}|.\]
\end{cor}

\begin{proof}
Since $A$ is normal, $\overline{\REAL A}$ and $\overline{\Ima A}$
are strongly commuting self-adjoint operators. Hence
$\overline{\REAL A}$ and $i\overline{\Ima A}$ are strongly commuting
normal operators. Now, apply Theorem \ref{abs A+B both unbd normal
TTHM}.
\end{proof}

\begin{cor}
Let $A$ and $B$ be normal operators where $B\in B(H)$. If $BA\subset
AB$, then
\[|A+B|\leq |A|+|B|.\]
\end{cor}

\begin{proof}Since $B\in B(H)$, the condition $BA\subset
AB$ amounts to the strong commutativity of $A$ and $B$. It is also
plain that $A+B$ is closed and this finishes the proof.
\end{proof}

As an interesting consequence, we obtain a characterization of
invertibility for the class of unbounded normal operators.

\begin{pro}\label{invertible normal unbounded prop}
Let $A$ be a normal operator. Then
\[A\text{ is invertible}\Longleftrightarrow |\overline{\REAL A}|+|\overline{\Ima A}|\text{ is invertible}.\]
Particularly, if $\lambda=\alpha+i\beta$, then
\[\lambda\in\sigma(A)\Longleftrightarrow |\overline{\REAL A}-\alpha I|+|\overline{\Ima A}-\beta I| \text{ is not invertible.}\]
\end{pro}

\begin{proof}Since $A$ is invertible, so is $|A|$. Hence by
Corollary \ref{triangle inequality normal real imag cor},
$|\overline{\REAL A}|+|\overline{\Ima A}|$ too is invertible.
Conversely, by Propositions \ref{PP} \& \ref{qq}
\[\left|\overline{\REAL A}\right|\leq |A| \text{ and } \left|\overline{\Ima A}\right|\leq |A|.\]
Hence, by the same arguments as in the previous proofs, we may
easily establish
\[\left|\overline{\REAL A}\right|+\left|\overline{\Ima A}\right|\leq 2|A|.\]
Thus, the invertibility of $|\overline{\REAL A}|+|\overline{\Ima
A}|$ implies the invertibility of $|A|$. Since the latter means that
the normal operator $A$ is left invertible, by
\cite{Dehimi-Mortad-2018} we get that $A$ is invertible.
\end{proof}

As an interesting consequence, we have the following.

\begin{cor}Let $A$ be a normal operator. Then
\[\sigma(A)\subset \sigma(\overline{\REAL A})+i\sigma(\overline{\Ima A})\]
where the sum of sets is defined in the usual fashion.
\end{cor}

\begin{proof}
Let $\lambda=\alpha+i\beta\in\sigma(A)$. Proposition \ref{invertible
normal unbounded prop} tells us that $|\overline{\REAL A}-\alpha
I|+|\overline{\Ima A}-\beta I|$ is not invertible. If either
$|\overline{\REAL A}-\alpha I|$ or $|\overline{\Ima A}-\beta I|$ is
invertible, then $|\overline{\REAL A}-\alpha I|+|\overline{\Ima
A}-\beta I|$ would be invertible! Therefore, both $|\overline{\REAL
A}-\alpha I|$ and $|\overline{\Ima A}-\beta I|$ are not invertible,
i.e. $\overline{\REAL A}-\alpha I$ and $\overline{\Ima A}-\beta I$
are not invertible. In other words, $\alpha\in
\sigma(\overline{\REAL A})$ and $\beta\in \sigma(\overline{\Ima
A})$. Consequently, $\lambda\in \sigma(\overline{\REAL
A})+i\sigma(\overline{\Ima A})$, as needed.
\end{proof}

We also obtain a very simple proof of the "realness" of the spectrum
of unbounded self-adjoint operators.

\begin{cor}
Let $A$ be a self-adjoint operator. Then $\sigma(A)\subset\R$.
\end{cor}

\begin{proof}
Let $\lambda\not\in\R$, i.e. $\lambda=\alpha+i\beta$ ($\alpha\in\R$,
$\beta\in\R^*$). Since $A-\alpha I$ is self-adjoint, it follows that
$A-\alpha I-i\beta I$ is normal. By the invertibility of $|\beta|I$,
it follows that of $|A-\alpha I|+|\beta I|$ ($\geq |\beta|I$). By
Proposition \ref{invertible normal unbounded prop}, this just means
that $A-\lambda I$ is invertible, that is, $\lambda\not\in
\sigma(A)$.
\end{proof}

Using a proof by induction, we can extend Theorem \ref{abs A+B both
unbd normal TTHM} to a finite family of normal operators.

\begin{thm}
Let $(A_i)_{i=1,\cdots,n}$ be pairwise strongly commuting normal
operators. Then
\[|\overline{A_1+A_2+\cdots+A_n}|\leq |A_1|+|A_2|+\cdots+|A_n|.\]
\end{thm}

\begin{cor}Let $(A_i)_{i=1,\cdots,n}$ be pairwise commuting normal
operators. If $B\in B(H)$ is hyponormal and $BA_i\subset A_iB$ for
$i=1,2,\cdots,n$, then
\[|\overline{A_1+A_2+\cdots+A_n}+B|\leq |A_1|+|A_2|+\cdots+|A_n|+|B|.\]
\end{cor}

Next, we pass to another triangle inequality.

\begin{thm}\label{thm triangle inequality MAIN 2}
Let $A$ be a normal operator and let $B\in B(H)$ be hyponormal. If
$BA\subset AB$, then
\[||A|-|B||\leq |A+B|.\]
\end{thm}

\begin{proof}
The proof is in essence fairly similar to that of Theorem \ref{thm
triangle inequality MAIN} and we omit some details. Clearly,
$|A|-|B|$ and $A+B$ are closed and
\[D(|A+B|)=D(A+B)=D(A)=D(|A|)=D(|A|-|B|)=D(||A|-|B||).\]
Let $x\in D(A)$. Then
\begin{align*}
\|||A|-|B||x\|^2&=\|(|A|-|B|)x\|^2\\
&=<(|A|-|B|)x,(|A|-|B|)x>\\
&=\||A|x\|^2+\||B|x\|^2-<|A|x,|B|x>-<|B|x,|A|x>\\
&=\||A|x\|^2+\||B|x\|^2-<|B||A|x,x>-<|A||B|x,x>\\
&=\||A|x\|^2+\||B|x\|^2-2<|A||B|x,x>\\
&=\||A|x\|^2+\||B|x\|^2-2<|A^*||B|x,x>\\
&=\||A|x\|^2+\||B|x\|^2-2<|A^*B|x,x>.
\end{align*}
Since $A^*B$ is closed and hyponormal, we have for all $x$:
\[|<A^*Bx,x>|\leq <|A^*B|x,x>\]
and so
\[-2<|A^*B|x,x>\leq -2|<A^*Bx,x>|\leq 2\REAL(<A^*Bx,x>).\]
Hence
\begin{align*}
\|||A|-|B||x\|^2&\leq \||A|x\|^2+\||B|x\|^2+2\REAL(<A^*Bx,x>)\\
&=\||A|x\|^2+\||B|x\|^2+<A^*Bx,x>+\overline{<A^*Bx,x>}\\
&=\||A|x\|^2+\||B|x\|^2+<A^*Bx,x>+<x,A^*Bx>\\
&=\||A|x\|^2+\||B|x\|^2+<Bx,Ax>+<Ax,Bx>\\
&=\|Ax\|^2+\|Bx\|^2+<Bx,Ax>+<Ax,Bx>\\
&=\|(A+B)x\|^2\\
&=\||A+B|x\|^2.
\end{align*}
Therefore,
\[\|||A|-|B||x\|\leq \||A+B|x\|,\]
i.e. we have shown that
\[||A|-|B||^2\leq |A+B|^2\]
or merely
\[||A|-|B||\leq |A+B|.\]
\end{proof}

\begin{cor}
Let $A$ be a normal operator and let $B\in B(H)$ be hyponormal. If
$BA\subset AB$, then
\[||A|-|B||\leq |A-B|.\]
\end{cor}

One may wonder whether we may adapt the proof of Theorem \ref{abs
A+B both unbd normal TTHM} to generalize the previous theorem to a
couple of unbounded and strongly commuting normal operators? The
first observation leads to an issue with domains which seems to be
an insurmontable difficulty unless we impose an extra condition. For
instance:

\begin{thm}\label{generalized triangle inequality TWO NOrmal THM}
Let $A$ and $B$ be two strongly commuting unbounded normal
operators. If $D(\overline{A+B})\subset D(\overline{|A|-|B|})$,
then,
\[\left|\overline{|A|-|B|}\right|\leq |\overline{A+B}|.\]
\end{thm}

\begin{proof}Most details will be omitted. Write
\[A=\int_{\C}zdE_A \text{ and } B=\int_{\C}z'dF_B.\]
Then
\[\left|\overline{|A|-|B|}\right|=\iint\limits_{\R^2}||z|-|z'||dE_{A,B}\]
and
\[|\overline{A+B}|=\iint\limits_{\R^2}|z+z'|dE_{A,B}.\]
Argue as in Theorem \ref{abs A+B both unbd normal TTHM} to obtain
the desired inequality:
\[\left|\overline{|A|-|B|}\right|\leq |\overline{A+B}|.\]
\end{proof}

\begin{rema}
The condition $D(\overline{A+B})\subset D(\overline{|A|-|B|})$ may
be dropped at the cost of assuming that $B$ (say) is $A$-bounded
with a relative bounded strictly less than one. In such case $A+B$
is closed. Since $|B|$ is also $|A|$-bounded, $|A|-|B|$ too is
closed. Everything else remains unchanged.
\end{rema}

We finish with a related inequality.

\begin{pro}
If $A$ is a normal operator, then
\[\left|\overline{|\overline{\REAL A}|-|\overline{\Ima A}|}\right|\leq |A|.\]
\end{pro}

\begin{proof}
We know that $||\REAL z|-|\Ima z||\leq |z|$ for any complex $z$.
Also,
\[|\overline{\REAL A}|=\int_{\R}|\REAL z|dE_{\overline{\REAL A}}\text{ and }|\overline{\Ima A}|=\int_{\R}|\Ima z|dF_{\overline{\Ima A}}.\]
Since $|\overline{\REAL A}|$ and $|\overline{\Ima A}|$ are commuting
self-adjoint operators thanks to the normality of $A$, we may form
the double integral defining $|\overline{|\overline{\REAL
A}|-|\overline{\Ima A}|}$. Therefore, we need only check the wanted
inclusion of domains (and then proceed as in the proof of Theorem
\ref{abs A+B both unbd normal TTHM}). Since $A$ is normal,
$D(A)=D(A^*)$. Hence
\[D(A)=D(\REAL A)\subset D(\overline{\REAL A})=D(|\overline{\REAL A}|) \text{ and } D(A)=D(|\overline{\Ima A}|)\]
and so
\[D(|A|)=D(A)\subset D(|\overline{\REAL A}|)\cap D(|\overline{\Ima A}|)=D\left(\left|\overline{|\overline{\REAL A}|-|\overline{\Ima A}|}\right|\right),\]
finishing the proof.
\end{proof}

\section{Questions}

Although we have more or less succeeded in dealing with the absolute
value of unbounded operators, some questions remain unfortunately
unanswered. For instance:

\begin{enumerate}
  \item If $A,B\in B(H)$ are hyponormal and commuting, then do we
  have
  \[|A+B|\leq |A|+|B|?\]
  This question is interesting in the sense that there are no
  counterexamples if $\dim H<\infty$ for it is shown that the
  previous inequality holds for commuting normal operators (and as
  is known the classes of normal and hyponormal operators coincide
  when $\dim H<\infty$).
  \item Can Theorem \ref{56} be weakened to the class of closed hyponormal
  operators without making further assumptions?
  \item As observed above, can we prove
  \[\left|\overline{|A|-|B|}\right|\leq |\overline{A+B}|\]
  by solely assuming that $A$ and $B$ are strongly commuting normal
  operators?
\end{enumerate}


\begin{thebibliography}{1}

\bibitem{Arai-2005-Rev-math-phys}
A. Arai, \textit{Generalized weak Weyl relation and decay of quantum
dynamics}, Rev. Math. Phys., \textbf{17/9} (2005) 1071-1109.

\bibitem{Bernau JAusMS-1968-square root}
S.~J.~Bernau, \textit{The square root of a positive self-adjoint
operator}, J. Austral. Math. Soc., \textbf{8} (1968) 17-36.

\bibitem{Birman-Solomjak-sepctral theorem of s.a. operators}
M. Sh. Birman, M. Z. Solomjak, Spectral theory of selfadjoint
operators in Hilbert space. Translated from the 1980 Russian
original by S. Khrushch\"{e}v and V. Peller. Mathematics and its
Applications (Soviet Series). \textit{D. Reidel Publishing Co.},
Dordrecht, 1987.

\bibitem{Con}
J. B. Conway, A course in functional analysis, Graduate Texts in
Mathematics, \textbf{96}. \textit{Springer-Verlag}, 1990 (2nd
edition).

\bibitem{Dehimi-Mortad-2018}
S. Dehimi, M. H. Mortad, \textit{Right (or left) invertibility of
bounded and unbounded operators and applications to the spectrum of
products}, Complex Anal. Oper. Theory, \textbf{12/3} (2018) 589-597.

\bibitem{Dehimi-Mortad-REID}
S. Dehimi, M. H. Mortad, \textit{Generalizations of Reid
Inequality}, Mathematica Slovaca (to appear). arXiv:1707.03320v1.

\bibitem{Fong-Tsui 1981}
C. K. Fong, S. K. Tsui,\textit{ A note on positive operators}, J.
Operator Theory, {\bf 5/1}, (1981) 73-76.

\bibitem{Gesztesy-Schm\"{u}dgen-2018}
F. Gesztesy, K. Schm\"{u}dgen, \textit{Some remarks on the operator
$T^*T$} (2018). arXiv:1802.05793v2.

\bibitem{Gustafson-Mortad-2016}
K. Gustafson, M. H. Mortad, \textit{Conditions implying
commutativity of unbounded self-adjoint operators and related
topics}, J. Operator Theory, \textbf{76/1} (2016) 159-169.

\bibitem{Jablonski et al 2014}
Z. J. Jab{\l}o\'{n}ski, Il B. Jung, J. Stochel, \textit{Unbounded
quasinormal operators revisited}, Integral Equations Operator Theory
\textbf{79/1} (2014) 135-149.

\bibitem{JOCIC}
D. Joci\'{c}, \textit{A characterization of formally symmetric
unbounded operators}, Publ. Inst. Math. (Beograd) (N.S.),
\textbf{46(60)} (1989), 141-144.

\bibitem{Jung-Mortad-Stochel}
Il Bong Jung, M. H. Mortad, J. Stochel, \textit{On normal products
of selfadjoint operators}, Kyungpook Math. J., \textbf{57} (2017)
457-471.

\bibitem{Meziane-Mortad}
M. Meziane, M. H. Mortad, \textit{Maximality of linear operators},
(submitted). arXiv:1711.00521v1.

\bibitem{Mortad-PAMS2003}
M. H. Mortad, \textit{An application of the Putnam-Fuglede theorem
to normal products of selfadjoint operators}, Proc. Amer. Math. Soc.
{\bf 131} (2003) 3135-3141.


\bibitem{Mortad-Oper-TH-BOOK-WSPC}
M. H. Mortad, An Operator Theory Problem Book, \textit{World
Scientific Publishing Co.}, (to appear).

\bibitem{Mortad-abs-value-BOUNDED}
M. H. Mortad, \textit{On the absolute value of the product and the
sum of linear operators} (submitted). arXiv:1702.08671.

\bibitem{Mortad-Fong-TSUI-VERY-FIRST-PAper}
M. H. Mortad, \textit{A contribution to the Fong-Tsui conjecture
related to self-adjoint operators}. arXiv:1208.4346.

\bibitem{Putnam-book-1967}
C. R. Putnam, Commutation properties of Hilbert space operators and
related topics, \textit{Springer-Verlag}, New York, 1967.

\bibitem{SCHMUDG-book-2012}
K. Schm\"{u}dgen, Unbounded self-adjoint operators on Hilbert space,
\textit{Springer} GTM {\bf 265}  (2012).

\bibitem{Sebestyen-Tarcsay-TT* von Neumann T closed}
Z. Sebestyén, Z. Tarcsay, \textit{A reversed von Neumann
theorem}, Acta Sci. Math. (Szeged), \textbf{80/3-4} (2014) 659-664.

\bibitem{Sebestyen-Tarcsay-square root-NEW}
Z. Sebestyén, Z. Tarcsay, \textit{On the square root of a
positive selfadjoint operator}, Period. Math. Hungar., 75/2 (2017)
268-272.

\bibitem{Simon-Operator-Theory-NEW}
B. Simon, Operator theory. A Comprehensive course in analysis, Part
4. \textit{American Mathematical Society}, Providence, RI, 2015.

\bibitem{Tian-Jorgensen PhD THESIS}
F. Tian, On commutativity of unbounded operators in Hilbert space.
Thesis (Ph.D.) \textit{The University of Iowa} (2011).
http://ir.uiowa.edu/etd/1095/.

\bibitem{WEI}
J.~Weidmann, Linear operators in Hilbert spaces (translated from the
German by J. Sz\"{u}cs), \textit{Springer-Verlag}, GTM {\bf 68}
(1980).

\end{thebibliography}
\end{document}